\documentclass{lms}
\usepackage{amsmath, amssymb,amsfonts}
\usepackage{theorem}
\usepackage{url}

\newtheorem{theorem}{Theorem}[section]

\newtheorem{proposition}[theorem]{Proposition}
\newtheorem{lemma}[theorem]{Lemma}

\newtheorem{definition}[theorem]{Definition}
\newtheorem{remark}[theorem]{Remark}

\DeclareMathOperator{\chr}{char}

\DeclareMathOperator{\im}{im}
\DeclareMathOperator{\Jac}{Jac}
\DeclareMathOperator{\Sym}{Sym}

\newcommand{\ignore}[1]{}

\newcommand{\frakp}{\mathfrak{p}}
\newcommand{\tr}[1]{\tau_{#1}}
\newcommand{\oo}{\mathcal{O}}
\newcommand{\FF}{\mathbb{F}}
\newcommand{\PP}{\mathbb{P}}

\newcommand{\conj}[1]{\overline{#1}}
\newcommand{\card}[1]{|#1|} 
\newcommand{\isom}{\cong} 
\newcommand{\system}{set}

\title[]{Complete addition laws on abelian varieties}
\author{Christophe Arene, David Kohel and Christophe Ritzenthaler}

\classno{11G20, 14K15, 11T71}

\extraline{The authors acknowledge the financial support by 
grant ANR-09-BLAN-0020-01 from the French ANR, 
the AXA Research Fund for the PhD grant of the first author 
and of the grant MTM2006-11391 from the Spanish MEC for the third author}

\begin{document}

\maketitle

\begin{abstract}
We prove that under any projective embedding of an abelian variety $A$ of 
dimension $g$, a complete set of addition laws has cardinality at least 
$g+1$, generalizing of a result of Bosma and Lenstra 
for the Weierstrass model of an elliptic curve in $\PP^2$.  
In contrast,  we moreover prove that if $k$ is any field with infinite absolute Galois 
group, then there exists, for every abelian variety $A/k$, a projective 
embedding and an addition law defined for every pair of $k$-rational points.  
For an abelian variety of dimension 1 or 2, we show that this embedding can be 
the classical Weierstrass model or the embedding in $\PP^{15}$, respectively, 
up to a finite number of counterexamples for $|k| \le 5$.
\end{abstract}

\section{Introduction}

The notion of completeness of a \system\ of addition laws for an abelian variety 
$A$ in $\PP^r$ was introduced by Lange and Ruppert~\cite{LaRu}. We recall that 
an addition law is an $(r+1)$-tuple of bihomogeneous polynomials $(p_0,\dots,p_r)$ 
such that the map
$$
(x,y) \longmapsto (p_0(x,y),\dots,p_r(x,y)),
$$
determines the group law $\mu: A \times A \rightarrow A$ on an open subset of 
$A \times A$, and a set of addition laws is complete if these open sets cover 
$A \times A$ (see Definition~\ref{def:Complete}).
The bidegree $(m,n)$ of an addition law is the bidegree of the polynomials $p_i$ in 
$x$ and $y$. Lange and Ruppert prove that the minimal bidegree of any addition 
law is $(2,2)$ and determine exact dimensions for the spaces of all addition 
laws of given bidegree.  For an elliptic curve $E$ in $\PP^2$ in Weierstrass 
form, the space of addition laws has dimension $3$, and Bosma and 
Lenstra~\cite{BosmaLenstra} proved that two  
suffice for a complete \system, determining $\mu$ on all of $E \times E$. 

In 2007, Edwards introduced a new normal form for elliptic curves
$$
x_1^2 + x_2^2 = a^2(1 + x_1^2 x_2^2),
$$
with particularly simple rational expression for the group law.
After a coordinate scaling, Bernstein and Lange~\cite{BLfasteraddition} descend 
this model to 
$$
x_1^2 + x_2^2 = 1 + d x_1^2 x_2^2,
$$
for $d = a^4$, which admits the group law $x + y = z$  
where 
$$
z = \left(
  \frac{x_1 y_2 + x_2 y_1}{1 + d x_3 y_3},
  \frac{y_3 - x_3}{1 - d x_3 y_3}\right), 
$$
for $x_3 = x_1 x_2$ and $y_3 = y_1 y_2$. 
In addition to giving a precise analysis of the efficiency of this group law, 
Bernstein and Lange observe that the addition law is $k$-complete over any 
field $k$ in which $d$ is a nonsquare (i.e.~the addition law is well-defined 
on all pairs of $k$-rational points of $E$). 
To interpret these rational expressions in terms of projective addition 
laws as analyzed by Lange and Ruppert, we note that $\{1,x_1,x_2,x_3\}$ 
forms a basis of global sections for the Riemann--Roch space of the divisor 
at infinity for the pair of coordinate functions $(x_1,x_2)$, and that this 
basis determines a projective embedding 
$$
(x_1,x_2,x_3) \longmapsto (1:x_1:x_2:x_3)
$$
in $\PP^3$ which is projectively normal (see Section 2 for precise definitions).  
Namely the image curve is of the form
$$
X_1^2 + X_2^2 = X_0^2 + dX_3^2,\ X_0 X_3 = X_1 X_2.
$$
The Edwards addition law can be interpreted as the  
bidegree $(2,2)$ addition law
$$
\begin{array}{r@{\,}r@{\;}l}
\big( 
  & (X_0 Y_0 + d X_3 Y_3) (X_0 Y_0 - d X_3 Y_3),
  & (X_0 Y_0 - d X_3 Y_3) (X_1 Y_2 + X_2 Y_1), \\
  & (X_0 Y_0 + d X_3 Y_3) (X_0 Y_3 - X_3 Y_0),
  & (X_0 Y_3 - X_3 Y_0) (X_1 Y_2 + X_2 Y_1) \,\big).
\end{array}
$$
Any elliptic curve specified by an affine model has a canonical embedding associated 
to the complete linear system. Consequently, we refer only to such abelian varieties 
with projective embeddings.  

In terms of degree 3 models, Bernstein, Kohel and Lange~\cite{BKLHessian} 
construct a $k$-complete addition law on the family of twisted Hessian 
curves
$$
aX_0^3 + X_1^3 + X_2^3 = d X_0 X_1 X_2,
$$
which admit the $k$-complete addition laws
$$
(
  X_0 X_1 Y_1^2 - X_2^2 Y_0 Y_2,\ 
  a X_0 X_2 Y_0^2 - X_1^2 Y_1 Y_2,\  
  -a X_0^2 Y_0 Y_1 + X_1 X_2 Y_2^2
),
$$
and 
$$
(
  X_0 X_2 Y_2^2 -X_1^2 Y_0 Y_1,\
  -a X_0^2 Y_0 Y_2 + X_1 X_2 Y_1^2,\ 
  a X_0 X_1 Y_0^2 - X_2^2 Y_1 Y_2
),
$$
over any field $k$ in which $a$ is not a cube.  Any such model is equivalent 
to a Weierstrass model by a linear change of variables, which shows that the 
property of $k$-completeness is not special to quartic models in $\PP^3$. 

Both the Edwards and twisted Hessian models share the property that they 
require a level structure of rational torsion.  In analogy with the quartic 
Edwards model, Bernstein and Lange~\cite{BL-UglyComplete} demonstrate by 
example that a general elliptic curve admits a quartic model with $k$-complete 
addition law (subject to some coefficient being a nonsquare), while resorting 
to a rational expression for an addition law of high bidegree.  
The second author of the present article gives an elementary characterization 
of $k$-completeness of addition laws of bidegree $(2,2)$ in terms of the Galois 
action on an associated divisor on the curve~\cite[Corollary~12]{Kohel-AdditionLaws}.  
In particular, the property of $k$-completeness on elliptic curves is not special.

In this paper, we generalize the above results to abelian varieties. We determine 
new, tight bounds on the size of a complete \system\ of addition laws under any 
embedding, a generalization of the result of Bosma and Lenstra~\cite{BosmaLenstra} 
for elliptic curves.  Moreover we prove that if $k$ is any field with infinite 
absolute Galois group, then there exists, for every abelian variety $A/k$, a 
projective embedding and an addition law defined for every pair of $k$-rational 
points (see Theorem~\ref{thm:AbVar}). 

Our work builds on the elegant paper of Lange and Ruppert~\cite{LaRu}, 
in which the authors interpret addition laws on an abelian variety $A/k$ 
in terms of sections of a certain line bundle $\mathcal{M}$ on $A\times A$. 
Our key idea is to observe that an addition law associated to a section $s$ 
of $H^0(A\times A,\mathcal{M})$ with zero divisor $D_s := (s)_0$ is defined on 
$A\times A\setminus D_s$. We obtain a $k$-complete addition law by constructing 
a $k$-rational divisor $D_s$ without any $k$-rational point. 
This gives an exact analog of the elliptic curve case studied by the second 
author~\cite{Kohel-AdditionLaws}.

In Section 2, we recall some definitions and concepts of \cite{LaRu}, explain 
more explicitly the link between addition laws on a projective embedding of $A/k$ 
and sections of $H^0(A\times A,\mathcal{M})$, and also deal with the geometric 
case $k=\bar{k}$.  For any principally polarized abelian variety of dimension $g$, 
we give bounds on the cardinality of any complete \system\ of addition laws. 
In particular we show that its cardinality is at least $g+1$.

In Section 3, we consider the case of a field $k$ with infinite absolute 
Galois group, and prove the aforementioned result on existence of a pair 
consisting of a projective embedding and a $k$-complete addition law. 

In Section 4, we specialize to elliptic curves and Jacobians of genus 2 curves over 
a finite field $k$, noting that the results also extend to other fields (see 
Remarks \ref{rem:HilbertianFields} and \ref{rem:NumberFields}). We prove that there 
exists a $k$-complete addition law for their classical embeddings in $\PP^2$ and 
$\PP^{15}$, respectively, as soon as $\card{k} \geq 5$ for elliptic curves and 
$\card{k} \geq 7$ for Jacobian surfaces. In particular, we exhibit an explicit 
$k$-complete addition law on a Weierstrass model of an elliptic curve $E$ over 
$k$ when $E$ has no nontrivial rational $2$-torsion point.

\section{Addition laws and completeness}
\label{sec:addition}

Let $k$ be a field and $A/k$ be an abelian variety of dimension $g$. We assume that $A$ is embedded in some projective
space $\PP^r$ over $k$, by a very ample line bundle $\mathcal{L}=\mathcal{L}(D)$ for $D$ an effective divisor, and we
denote by $\iota : A \hookrightarrow \PP^r$ the corresponding morphism. We also assume in the sequel that the embedding
is projectively normal. Recall that $A$ is said to be \emph{projectively normal} in $\PP^r$ if for every $n\ge1$ the restriction map $\Gamma(\PP^r,\mathcal{O}_{\PP^r}(n))\rightarrow\Gamma(A,\mathcal{L}^n)$ is surjective. This is the case in the classical settings where $\mathcal{L}=\mathcal{L}_0^a$ with $\mathcal{L}_0$ an ample line bundle and $a\ge 3$ \cite[p.187]{bilange}.

Let $I_1$ and $I_2$ be the homogeneous defining ideal for $A$ in $k[X_0,\ldots,X_r]$ and 
$k[Y_0,\ldots,Y_r]$, respectively. The \emph{group law} 
$$
\mu : A\times A\rightarrow A,
$$
defined by $(x,y) \mapsto x+y$, can be locally described by bihomogenous polynomials. 
\ignore{
More precisely, an \emph{addition law} $\frakp$ of bidegree $(m,n)$ on $\iota(A) \subset \PP^r$ 
is a nonempty open set $U\subset A \times A$ together with $r+1$ elements 
$$
p_0,\ldots,p_r \in k[X_0,\ldots,X_r]/I_1\otimes k[Y_0,\ldots, Y_r]/I_2,
$$ 
bihomogeneous of degree $m$ in $X_0,\ldots,X_r$ and of degree $n$ in $Y_0,\ldots,Y_r$ 
such that we have 
$$
\iota\circ\mu\,(x,y) = \Big(p_0\big(\iota(x),\iota(y)\big) : \ldots : p_r\big(\iota(x),\iota(y)\big)\Big)
$$
for all $(x,y) \in U(\bar{k})$.
}
More precisely, an \emph{addition law} $\frakp$ of bidegree $(m,n)$ on $\iota(A) \subset \PP^r$ 
is an $(r+1)$-tuple $(p_0,\ldots,p_r)$ of elements 
$$
p_i \in k[X_0,\ldots,X_r]/I_1 \otimes k[Y_0,\ldots, Y_r]/I_2,
$$ 
which are bihomogeneous of degree $m$ and $n$ in $X_0,\ldots,X_r$ and $Y_0,\ldots,Y_r$, 
respectively, and for which there exists a nonempty open subset $U$ of $A \times A$ 
such that, for all $(x,y) \in U(\bar{k})$,
$$
\iota\circ\mu\,(x,y) = \big(p_0(\iota(x),\iota(y)) : \ldots : p_r(\iota(x),\iota(y))\big).
$$
When $A$ is given with a fixed embedding in $\PP^r$ we may suppress the 
reference to the embedding $\iota$ and speak of addition laws on $A$. 

\begin{definition}
\label{def:Complete}
 A set $S$ of addition laws is said to be \emph{$k$-complete} if for any $k$-rational point $(x,y) \in (A \times A)(k)$ there is an
 addition law in $S$ defined on an open set $U$ containing $(x,y)$. This set is said to be \emph{complete} if the
 previous property is true over $\bar{k}$. If $S=\{\frakp\}$ is a singleton, we say the addition law $\frakp$ is
 \emph{$k$-complete} and \emph{complete} when $k=\bar{k}$. 
 \end{definition}
 
In \cite[Lem.2.1]{LaRu}, Lange and Ruppert give the interpretation of the possible addition laws in terms of the sections of  certain line bundles.

\begin{proposition}
\label{prop:addition-law-sheaf}
Let $\pi_1,\pi_2 : A \times A \to A$ be the projection maps on the first and second factor. 
There is an addition law (respectively a  complete set of addition laws) of bidegree $(m,n)$ 
on $A$ with respect to the embedding in $\PP^r$ determined by $\mathcal{L}$ if and only if 
$$
H^0(A \times A, \mathcal{M}_{m,n}) \neq 0
$$ 
(respectively the linear system $\mid\mathcal{M}_{m,n}\mid$ is basepoint-free), 
where $$
\mathcal{M}_{m,n}=\mu^*\mathcal{L}^{-1} \otimes \pi_1^*\mathcal{L}^{m} \otimes \pi_2^*\mathcal{L}^{n}.
$$
\end{proposition} 

We explain how one associates an addition law to a nonzero section $w$ in $H^0(A \times A,\mathcal{M}_{m,n})$. 
For $0 \leq j \leq n$, let $t_j \in H^0(A,\mathcal{L})$ be the basis given by $t_j=\iota^* X_j$ where $X_j$ are 
the coordinate functions on $\PP^r$. As shown in \cite[p.607]{LaRu}, $H^0(A \times A,\mu^* \mathcal{L}) = 
\mu^*H^0(A,\mathcal{L})$, so $s_j= \mu^* t_j$ is a basis of $H^0(A \times A,\mu^* \mathcal{L})$. For each $j$ 
and $(x,y) \in A \times A$, we have 
$$
s_j(x,y)=t_j \circ \mu\,(x,y)=X_j (\iota\circ\mu\,(x,y)).
$$
Now $w \otimes s_j \in H^0(A \times A,\pi_1^* \mathcal{L}^{m} \otimes \pi_2^* \mathcal{L}^{n})$. 
As the embedding is projectively normal we have 
$$
\pi_1^* \mathcal{L}^{m} \otimes \pi_2^*\mathcal{L}^{n}=(\iota\otimes \iota)^*\oo_{\PP^r}(m) \otimes \oo_{\PP^r}(n),
$$ 
then there exists a bihomogeneous polynomial $p_j$ of bidegree $(m,n)$ such that for all points $(x,y) \in A \times A$
$$
(w \otimes s_j)(x,y)=p_j(\iota(x),\iota(y)).
$$
Therefore, if $U=A \times A \setminus (w)_0$, we have
$$
\begin{array}{r@{\ }c@{\ }l}
(p_0(\iota(x),\iota(y)) : \ldots : p_r(\iota(x),\iota(y))) & = & ((w \otimes s_0)(x,y) : \ldots : (w \otimes s_r)(x,y)) \\
&=& (s_0(x,y) : \ldots : s_r(x,y)) \\
&=& (X_0(\iota\circ\mu(x,y)) : \ldots : X_r(\iota\circ\mu(x,y)))\\
&=& \iota(\mu(x,y)).
\end{array}
$$
Another natural requirement to ask is that $\mathcal{L} = \mathcal{L}(D)$ 
 be symmetric, i.e.~$[-1]^*\mathcal{L} \isom \mathcal{L}$, or 
equivalently $D \sim [-1]^*D$, as we can see in the following lemmas.
\begin{lemma}
\label{lem:Inversion}
If $A/k$ is embedded in $\PP^r$ by a very ample symmetric line bundle $\mathcal{L}$ (projectively normal), then the inversion map $[-1]$ on $A$ is induced by a linear automorphism of $\PP^r$. Moreover if $\chr(k)\not=2$ there is a choice of coordinates such that the inversion acts by $\pm1$ on each coordinate.
\end{lemma}
\begin{proof}
The first statement is a direct consequence of the symmetry of $\mathcal{L}$. 
Now fix a basis $(t_i)$ of $H^0(A,\mathcal{L})$ and let $M$ be 
the matrix of the coordinates of $[-1]^*t_i$ in the basis $(t_i)$. 
The morphism $[-1]$ is induced by an involution of $\PP^r$ so there exists 
$\varepsilon\in k$ such that $M^2 - \varepsilon Id = 0$. 

The neutral element $O=(a_0:\cdots:a_r)$ of $A \hookrightarrow \PP^r$  
is a fixed point for~$[-1]$. Hence, the vector $(a_0,\ldots,a_r)$ is 
an eigenvector of the matrix $M$ with eigenvalue $\varepsilon_0 \in k$. 
This implies that $\varepsilon = \varepsilon_0^2$ and if $\chr(k)\not=2$ 
then $M^2-\varepsilon Id$ factors as $(M-\varepsilon_0 Id)(M+\varepsilon_0 Id)$. 
This proves that $M$ can be diagonalized over $k$ with eigenvalues in 
$\{\pm\epsilon_0\}$ and the conclusion holds.
\end{proof}

Before considering non-algebraically closed fields, it is natural to consider 
what happens over $\bar{k}$. We start by giving an upper bound on the cardinality 
of a complete set of addition laws.  In what follows we define the difference map
$
\delta : A \times A \longrightarrow A
$
by $(x,y) \mapsto x-y$, and use the product partial order on bidegree given by 
$(k,l) \le (m,n)$ if and only if $k \le m$ and $l \le n$. 
For bidegree $(m,n) = (2,2)$ we denote the line bundle $\mathcal{M}_{m,n}$ 
of Proposition~\ref{prop:addition-law-sheaf} by $\mathcal{M}$.  
We begin by recalling a fundamental lemma of Lange and Ruppert~\cite[Prop.2.2, Prop.2.3]{LaRu}.
 
\begin{lemma}
\label{lem:image}
Let $\mathcal{L}$ be an ample line bundle on $A$.
\begin{enumerate}
\item if $\mathcal{L}$ is not symmetric then $H^0(A \times A,\mathcal{M})= 0$, and 
\item if $\mathcal{L}$ is symmetric then $\mathcal{M}$ is isomorphic to $\delta^* \mathcal{L}$ 
and is basepoint-free, and consequently $h^0(\mathcal{M})=h^0(\mathcal{L})$.
\end{enumerate} 
If $(m,n) > (2,2)$ then $h^0(\mathcal{M}_{m,n})=h^0(\mathcal{L})^2(mn-m-n)^g$.
\end{lemma}
\begin{proof}
For $(m,n) > (2,2)$, the proof follows the case $(m,n)=(2,3)$ treated in \cite[Prop.2.3]{LaRu}.
For $(m,n) = (2,2)$, Lange and Ruppert prove in \cite[Proposition 2.2]{LaRu} that 
$\mathcal{M} \isom \delta^*\mathcal{L}$ and that $\mathcal{M}$ is basepoint-free. 
The equality $h^0(\mathcal{M})=h^0(\mathcal{L})$ is an easy consequence of the fact 
proved in \textit{loc.~cit.} that $\mathcal{M}|_{K(\mathcal{M})_0}$ is trivial 
and of the fact that, as $\mathcal{L}$ is ample, its index is zero. 
Indeed, according to \cite[Theorem 1(ii) p.95]{kempf}, one then has the isomorphism 
$H^0(A\times A,\mathcal{M}) \isom H^0(A,\mathcal{L})$.
\end{proof}

The isomorphism of $\mathcal{M}$ with $\delta^*\mathcal{L}$ allows us to consider 
line bundles on $A$ instead of $A\times A$.  The following well-known lemma 
shows that we can always find a symmetric embedding of $A/\bar{k}$.

\begin{lemma}
Let $(A,\lambda)$ be a principally polarized abelian variety over $\bar{k}$. 
There exists a symmetric line bundle which induces the polarization $\lambda$ on~$A$.
\end{lemma}
\begin{proof}
Suppose that $\mathcal{L}'$ is a line bundle attached to the polarization $\lambda$. 
We construct a symmetric line bundle $\mathcal{L}$ algebraically equivalent to $\mathcal{L}'$. 
Since $\mathcal{L}'$ is algebraically equivalent to $[-1]^*\mathcal{L}'$ 
(see \cite[p.93]{lang-ab}), there exists $x \in A(\bar{k})$ such that the 
translation $\tr{x}^*\mathcal{L}'$ is algebraically equivalent to $[-1]^*\mathcal{L}'$.
Let $y$ be an element of $A(\bar{k})$ such that $2y = x$, and set $\mathcal{L} 
= \tr{y}^*\mathcal{L}'$. Then $\mathcal{L}$ is algebraically equivalent to 
$\mathcal{L}'$ and 
$$
\mathcal{L}
  = \tr{y}^* \mathcal{L}' 
  = \tr{-y}^*\tr{x}^*\mathcal{L}' 
  \isom \tr{-y}^*[-1]^*\mathcal{L}' 
  = [-1]^*\mathcal{L}, 
$$
hence symmetric.
\end{proof}

Suppose that $\mathcal{L}$ is a symmetric line bundle as in the preceeding 
lemma.  By Lemma~\ref{lem:image} the embedding defined by $\mathcal{L}^3$ 
has a complete \system\ of biquadratic addition laws of cardinality equal 
to $h^0\left(A,\mathcal{L}^3\right) = 3^g$.  This gives an upper bound on 
the minimal size of a complete \system\ of addition laws.  We now determine a 
lower bound.

\begin{theorem}
Assume $A$ is embedded in $\PP^r$ by a symmetric line bundle. 
If $S$ is a complete \system\ of addition laws on $A$ then $\card{S} \ge g+1$.
\end{theorem}

\begin{proof}
Suppose that $S$ is a complete \system\ of addition laws of bidegree $(m,n)$ 
on $A$, and let $\nabla = \ker(\mu) \subset A \times A$.  
By Lemma~\ref{lem:Inversion}, the isomorphism 
$$
[\,1\,] \times [-1] : A \longrightarrow \nabla 
$$
is linear, and so $([\,1\,] \times [-1])^*S$ is a set of polynomial 
(rational) maps for $A \rightarrow \{O\} \subset A$.  
It follows that there exists a set $I$ of polynomials of degree $m+n$ 
such that
$$
([\,1\,] \times [-1])^*S = \left\{ 
  \big(a_0 q(X_0,\dots,X_r),\dots,a_r q(X_0,\dots,X_r)\big) 
  \,:\, q \in I \right\},
$$
where $O = (a_0:\dots:a_r)$.  
Since $S$ is complete, the subvariety $V(I) \cap A$ is empty.
On the other hand, its dimension is at least $\dim(A) - \card{I} \geq g - \card{S}$, 
hence the cardinality of $S$ must be at least $g+1$. 
\end{proof}

Although the interval $[g+1, 3^g]$ is quite large, the lower bound shows that 
there is no complete addition law on any abelian variety of any dimension. 
For $g=1$, these bounds show that the minimal size of a complete set of addition 
laws is either $2$ or $3$.  An explicit set of cardinality $3$ was already 
given by Lange and Ruppert~\cite[Sec.3]{LaRu} if $\chr(k) \not= 2,3$, and 
in~\cite{LaRu2} for any characteristic, and Bosma and Lenstra~\cite{BosmaLenstra} 
proved that a set of minimal cardinality $2$ is in fact sufficient.

\section{$k$-complete addition laws}
\label{sec:k-complete}

Let $\mathcal{L}$ be a very ample symmetric line bundle defined 
by an effective $k$-rational divisor $D$ on $A/k$.  

Since $\delta^*\mathcal{L} \isom \mathcal{M}
= \mu^*\mathcal{L}^{-1} \otimes \pi_1^* \mathcal{L}^2 \otimes \pi_2^* \mathcal{L}^2$ 
there exists $w$ in $H^0(A\times A,\mathcal{M})$ such that $(w)_0 = \delta^*(D)$.
As we have seen in Section~\ref{sec:addition}, $w$  
defines a biquadratic addition law on the complement of $(w)_0=\delta^* D$. 
Hence it is sufficient that $D$ has no $k$-rational point for 
the group law to be $k$-complete. Note that this is also a necessary 
condition since a $k$-rational point $x$ on $D$ gives the $k$-rational 
point $(x,0)$ on $\delta^*D$. 

\begin{theorem} 
\label{thm:AbVar}
Let $A/k$ be an abelian variety and $\iota_0:A\hookrightarrow\PP^{r_0}$ be an embedding for some $r_0>1$. Assume that $k$ has
infinite absolute Galois group and let $d>r_0$ be such that there exists a separable extension $K/k$ of degree $d$ over $k$.  Then there exists an embedding $\iota:A\hookrightarrow\PP^r$ and a $k$-complete biquadratic addition law on $\iota(A)$, with $r=(2d)^g(r_0+1)-1$. 
\end{theorem}
\begin{proof}
Let $K = k(\alpha_0)/k$ be a separable extension, and denote by $\alpha_0,\ldots,\alpha_{d-1}$ 
its distinct Galois conjugates in the normal closure of $K/k$. 
For $i=0,\ldots,d-1$, let $H_{i}$ be the hyperplane in $\PP^{r_0}$
$$
H_{i} : X_{0}+\alpha_i X_{1}+\ldots+\alpha_i^{r_0} X_{r_0}=0.
$$

Since $d>r_0$, the sets $\{1,\alpha_i,\ldots,\alpha_i^{r_0}\}$ are linearly independent 
over $k$ for every $i$ and hence $H_{i}(k)$ is empty. 
Now $\sum H_{i}$ is a $k$-rational divisor, so let $D_0=\iota_0^*(\sum H_{i})$ and 
define the divisor $D=D_0+[-1]^*D_0$. Then $D$ is a symmetric, effective, $k$-rational divisor 
without $k$-rational points.  
Denote by $\mathcal{L}_0$ the line bundle associated to the embedding $\iota_0$. 
The line bundle $\mathcal{L}=\mathcal{L}(D)$ is isomorphic to $\mathcal{L}_0^{2d}$, 
so $\mathcal{L}$ is very ample and provides a projectively normal  embedding
$A\hookrightarrow\PP^{r}$ with a $k$-complete biquadratic addition law. 
By the Riemann-Roch theorem, the dimension $r$ is equal to $(2d)^g(r_0+1)-1$.
\end{proof}

\section{The genus one and two cases} 

In the previous section, a $k$-complete (biquadratic) addition law is proved 
to exist, for an embedding of the abelian variety in a projective space of 
high dimension. When $k=\FF_q$ is a finite field and the abelian variety $A/k$ 
has dimension $1$ or $2$, we will show that we can take the embedding to be 
the classical ones. In what follows, we let $\sigma$ denote the Frobenius 
automorphism of $\bar{k}/k$.

\subsection{Elliptic curves}
Let $A=E$ be an elliptic curve defined over $k=\FF_q$.
\begin{lemma}
\label{lem:AlignedOrbit}
If $q \ge 5$, there exists $P_0\in E(\bar{k})$ whose Galois orbit is given 
by three distinct points whose sum is $O$.
\end{lemma}
\begin{proof}
Consider the group homomorphism $N:E(\FF_{q^3})\rightarrow E(\FF_q)$ given by
$$
P \longmapsto P+P^{\sigma}+P^{\sigma^2}.
$$ 
We are looking for a point $P_0 \in \ker(N) \setminus E(\FF_q)$,  
hence we want 
$$
\card{\ker(N)} > \card{\ker(N) \cap E(\FF_q)}.
$$

The intersection of $\ker(N)$ with $E(\FF_q)$ is the group of $\FF_q$-rational 
$3$-torsion points of $E$ so $\card{\ker(N) \cap E(\FF_q)} \leq 9$. On the 
other hand, for all $q \ge 5$, we have 
$$
\card{\ker(N)} \ge \frac{\card{E(\FF_{q^3})}}{\card{E(\FF_q)}} \ge \frac{q^3+1-2\sqrt{q^3}}{q+1+2\sqrt{q}} > 9,
$$ 
so such a point $P_0$ exists in $E(\FF_{q^3})$.
\end{proof}

\begin{remark}
For each of $q = 2, 3$ and $4$, there exists at least one elliptic curve over 
$\FF_q$ for which $\card{\ker(N)} = \card{\ker(N) \cap E(\FF_q)}$.
\end{remark}

\begin{theorem}
\label{thm:PointsHorizontalLine}
Let $k$ be the finite field $\FF_q$ with $q\ge5$ and $E/k$ be an elliptic curve. There exists a $k$-complete biquadratic addition law on the Weierstrass model of $E\subset\PP^2$. 
\end{theorem}
\begin{proof}
Let $P_0$ be a point as in Lemma \ref{lem:AlignedOrbit} and $D$ be the divisor given by the sum of 
the Galois conjugates of $P_0$. It is a $k$-rational divisor without $k$-rational points. It is not 
a symmetric divisor but $\mathcal{L}=\mathcal{L}(D)$ is a symmetric line bundle as $D\sim3(O)\sim[-1]^*D$. 
Another consequence of the relation $D\sim3(O)$ is that the embedding associated to $\mathcal{L}(D)$ 
is projectively equivalent to the Weierstrass model of $E$.
\end{proof}

\begin{remark} 
\label{rem:HilbertianFields}
We use the fact that $k$ is a finite field only to prove the existence of the point 
$P_0$. It is easy to see that when $k$ is a number field, such a point always exists 
and so the conclusion of Theorem \ref{thm:PointsHorizontalLine} still holds. 
Indeed, if $E$ is defined by $y^2 + h(x) y = f(x)$, then, since $k$ is Hilbertian 
(see \cite[p.225]{LangFund}), there exists $y_0 \in k$ such that $y_0^2 + h(x) y_0 - f(x)$ 
is irreducible. We can take $P_0 = (x_0,y_0)$ where $x_0$ is any root of 
$y_0^2 + h(x) y_0 - f(x) = 0$ in $\bar{k}$.

In particular, for $\textrm{char}(k) \ne 2$ or $3$, by means of a change of variables 
we may assume $E$ is of the form $y^2 = x^3 + ax + b$.  Moreover, if $E$ has no non trivial 
$k$-rational $2$-torsion point, then the polynomial $f(x) = x^3 + ax + b$ is irreducible over $k$ 
and the sum $(X_1:Y_1:Z_1) + (X_2:Y_2:Z_2)$ is given by the addition law
$(X_3^{(2)},Y_3^{(2)},Z_3^{(2)})$ of Bosma and Lenstra~\cite{BosmaLenstra}: 
$$
\begin{array}{r@{}l}
\big(\,
% X3:
(X_1 Y_2 & \,+\, Y_1 X_2)(Y_1 Y_2 - 6b Z_1 Z_2) - a\, (Y_1 Z_2 + Z_1 Y_2) (2 X_1 X_2 - a Z_1 Z_2) \\
  &\,-\, X_1 Z_2 (a X_1 Y_2 + 3b Y_1 Z_2) - Z_1 X_2 (a Y_1 X_2 + 3b Z_1 Y_2),\\ 
% Y3:
Y_1^2 Y_2^2 
  &\,+\, a\, X_1 X_2 (3 X_1 X_2 - 2 a Z_1 Z_2) - a^2\, (X_1 Z_2 + Z_1 X_2)^2 \\
  &\,+\; 3b\, (X_1 Z_2 + Z_1 X_2)(3 X_1 X_2 - a Z_1 Z_2) - (a^3 + 9b^2) Z_1^2 Z_2^2,\\
% Z3:
Y_1 Y_2 
  & (Y_1 Z_2 + Z_1 Y_2) + (3 X_1 X_2 + 2a Z_1 Z_2)(X_1 Y_2 + Y_1 X_2) \\
  & +\; (a X_1 + 3b Z_1) Y_1 Z_2^2 + Z_1^2 (a X_2 + 3b Z_2) Y_2
\,\big),
\end{array}
$$
specialized to $(a_1,a_2,a_3,a_4,a_6) = (0,0,0,a,b)$. Under the hypothesis on the 
$2$-torsion, the exceptional divisor $\delta^*\{Y=0\}$ is irreducible, hence the 
addition law is $k$-complete. 
\end{remark}

\subsection{Genus $2$ curves} 
Let $C$ be a genus $2$ curve over a finite field $k=\FF_q$, with hyperelliptic 
involution $P \mapsto \conj{P}$.  By \cite[Proposition 2.3.21, p.180]{TV}, there 
exists a (not necessarily effective) $k$-rational divisor $P_\infty$ of degree 
$1$, such that $2 P_{\infty}$ 
is equivalent to the canonical divisor $\kappa$ of $C$.  The divisor $\Theta$, defined 
as the image of $C$ in $\Jac(C)$ under the map $P \mapsto (P)-P_{\infty}$, is   
then a $k$-rational, ample, symmetric divisor which defines the canonical 
principal polarization on $\Jac(C)$. For any $z \in \Jac(C)(\bar{k})$, we denote 
by $\Theta_z$ its translation $(\tr{z}^*)^{-1} \Theta=\Theta+z$.\\

The following result can be found for instance in \cite[p.275]{mumCJ}.
\begin{proposition}
\label{prop:decalage}
Let $P$ and $Q$ be points in $C(\bar{k})$, and set $z = (P)-(Q) 
\in \Jac(C)(\bar{k})$. Then we have 
$$
\Theta \cap \Theta_z = \left\{(P)-P_{\infty},\,(\conj{Q})-P_{\infty}\right\}.
$$
\end{proposition}

As in the previous section, we will need the existence of a Galois 
orbit of points for the construction of a divisor on $\Jac(C)$.
\begin{lemma} 
\label{lem:points}
If $q \geq 7$, there exists a point $P_0 \in C(\bar{k})$ whose Galois orbit 
has cardinality four and $P_0^{\sigma^2} = \conj{P}_0$. 
\end{lemma}
\begin{proof}
Let $\phi: C \rightarrow \PP^1$ be the quotient by the hyperelliptic involution.
Note that $P_0$ is a point in $C(\FF_{q^4}) \setminus C(\FF_{q^2})$ such that 
$\phi(P_0)$ is in $\PP^1(\FF_{q^2})$.  Moreover, no such point exists if and only 
if $\phi(C(\FF_{q^2})) = \PP^1(\FF_{q^2})$, or equivalently if   
$$
\card{C(\FF_{q^2})} = 2 (q^2+1) - e_2,
$$
where $e_2 \le 6$ is the number of ramification points of $\phi$ in $C(\FF_{q^2})$.
For $q \ge 7$, this equality contradicts the Weil bound $\card{C(\FF_{q^2})} \leq 
q^2 + 4q + 1$, and such a point exists. 
\end{proof}

\begin{remark}
For each $q = 2, 3, 4$ and $5$, there exists at least one genus $2$ curve 
over $\FF_q$ with no such point $P_0$.  In particular, for $q = 5$, the 
bound is tight (for $e_2 = 6$):
$$
\card{C(\FF_{q^2})} = 2 (q^2+1) - 6 = q^2 + 4q + 1 = 46,
$$
and is satisfied for the curve $y^2 = x^6+1$ over $\FF_5$. %(see~\cite{manYPoints}).

\end{remark}

\begin{theorem}
\label{thm:Genus2}
Let $C$ be a genus 2 curve over $\FF_q$ with $q \geq 7$. 
There exists a $k$-complete biquadratic addition law for the 
classical embedding of $\Jac(C)$ in $\PP^{15}$ determined by 
$\mathcal{L}(4\Theta)$. 
\end{theorem}
\begin{proof}
For the canonical divisor $\kappa$ and a point $P_0$ as in 
Lemma~\ref{lem:points}, we define 
$$
\begin{array}{ll}
\alpha_0 = (P_0)+(P_0^{\sigma}) - \kappa, &
\alpha_1 = (P_0^{\sigma})+(\conj{P}_0) - \kappa, \\
\alpha_2 = (\conj{P}_0)+(\conj{P}_0^{\sigma}) - \kappa & 
\alpha_3 = (\conj{P}_0^{\sigma})+(P_0) - \kappa. 
\end{array}
$$
Using Proposition~\ref{prop:decalage}, we find 
$$
\Theta_{\alpha_0}\cap \Theta_{\alpha_1} 
= (\tr{\alpha_0}^*)^{-1}\left(\Theta\cap \Theta_{(\conj{P}_0)-(P_0)}\right) 
= \left\{ (\conj{P}_0)-P_{\infty} + \alpha_0 \right\},
$$
$$
\Theta_{\alpha_0}\cap \Theta_{\alpha_3} 
= (\tr{\alpha_0}^*)^{-1}\Big(\Theta\cap \Theta_{(\conj{P}_0^{\sigma})-(P_0^{\sigma})}\Big) 
= \left\{ (\conj{P}_0^{\sigma})-P_{\infty} + \alpha_0 \right\}.
$$
By construction, the divisor $D = \sum \Theta_{\alpha_i}$ is ample, symmetric 
and $k$-rational. Moreover, since there exists a transitive action on the 
components $\Theta_{\alpha_i}$, any $k$-rational point of $D$ must be a 
point of the intersection
$$
\Theta_{\alpha_0} \cap \Theta_{\alpha_1} \cap
\Theta_{\alpha_2} \cap \Theta_{\alpha_3},
$$
which is empty.  
Finally, we have $\sum \alpha_i = 0$ by construction, so $D \sim 4 \Theta$ 
and $D$ determines a $k$-complete addition law for the classical 
embedding of $\Jac(C)$ in $\PP^{15}$ determined by $\mathcal{L}(4\Theta)$.
\end{proof}

\begin{remark} 
\label{rem:NumberFields}
This construction can be generalized to other fields. 
For instance, following the same lines as Remark~\ref{rem:HilbertianFields},
Lemma~\ref{lem:points} has an analogue over number fields $k$. 
However, a $k$-rational divisor $P_{\infty}$ of degree $1$ may 
no longer exist,  
but for the family of curves $C$ such as $y^2 = f(x)$ with 
$\deg f = 5$, we can take $P_{\infty}$ to be the divisor with support the point at infinity.  
In this case, the analogue of Theorem~\ref{thm:Genus2} holds over 
a number field. Arene and Cosset have developed an algorithm 
to construct such an addition law~\cite{AreneCosset}.
\end{remark}

\begin{remark}
The construction of Theorem \ref{thm:Genus2} uses differences of 
effective divisors of degree $g = 2$. In general such degree $g$
divisors are necessary, since if $C$ is a curve of genus $g$ 
and if we define $W_i=\im(\Sym^i C \rightarrow \Jac(C))$, then 
by \cite[p.~146]{farkaskra} the intersection 
$$
\bigcap\; \{ W_{g-1}-a \;:\; a\in W_r+b \}
$$ 
is nonempty for any $0\le r \le g-1$ and any $b \in \Jac(C)$.
\end{remark}

\begin{acknowledgements}The authors thank David Gruenewald 
for careful reading and comments on a prior version of this article.
\end{acknowledgements}

\end{document}